\documentclass[11pt]{amsart}

\usepackage{amsmath,amssymb,amsthm}
\usepackage{mathtools}
\usepackage{enumitem}
\usepackage{graphicx}
\usepackage{tikz}
\usepackage[margin=1in]{geometry}
\usepackage[hidelinks]{hyperref}

\theoremstyle{plain}
\newtheorem{theorem}{Theorem}[section]
\newtheorem{proposition}[theorem]{Proposition}
\newtheorem{lemma}[theorem]{Lemma}

\newtheorem{definition}[theorem]{Definition}
\newtheorem{example}[theorem]{Example}
\newtheorem{remark}[theorem]{Remark}

\newcommand{\Dow}{\mathrm{D}}
\newcommand{\DR}{\mathrm{DR}}

\newcommand{\Flag}{\mathrm{F}}

\newcommand{\R}{\mathbb{R}}

\title[Geometry-Induced Hodge Stars on Rips-Type Complexes]
      {Geometry-Induced Hodge Stars on Rips and Dowker--Rips Complexes}

\author{Jiahui Chen}
\address{Department of Mathematical Sciences, University of Arkansas, Fayetteville, AR, USA}
\email{jiahuic@uark.edu}

\author{Sebastian Wilcox}
\address{Department of Mathematical Sciences, University of Arkansas, Fayetteville, AR, USA}
\email{sjw032@uark.edu}

\date{\today}

\begin{document}

\begin{abstract}
    The Vietoris--Rips complex $\mathrm{VR}_\epsilon(X)$, the Dowker complex $\Dow_R(X,Y)$, and its flagified Dowker--Rips variant $\DR_R(X,Y)=\Flag(\Dow_R(X,Y))$ are simplicial complexes constructed from metric data or witness relations. They are useful in topological data analysis because they encode topology through combinatorial data derived from pairwise information, but at a fixed scale they retain little of the underlying geometry. Unlike alpha complexes or mesh-based discretizations, Rips-type complexes carry no canonical primal--dual cell structure, which is the ingredient used by the discrete exterior calculus Hodge star to encode metric information. We address this gap by equipping a Rips-type complex $K$ with diagonal geometry-induced Hodge stars represented by positive simplex weights $W_k=\operatorname{diag}\{w_k(\sigma):\sigma\in K_k\}$, which define weighted inner products on $k$-cochains. The resulting weighted discrete Hodge Laplacian $\Delta_k^W$ has kernel dimension equal to the $k$th Betti number of the underlying complex, while its nonzero spectrum is governed by the chosen geometric weights. The central issue is therefore not the existence of a weighted Laplacian, since any positive diagonal weights define one, but the design of weights that encode meaningful metric or witness geometry. We focus on two computable choices: simplex-volume weights, based on Euclidean simplex volumes, and soft witness weights, based on a Dowker-style support function $s_t(\sigma;Y)$ that quantifies higher-order witness support lost under flagification. We prove positivity, weighted self-adjointness, and Betti-number preservation for arbitrary positive diagonal weights, establish an asymptotic decay-rate characterization for soft witness support, and describe spectral descriptors derived from $\Delta_k^W$ for comparing geometry-aware Hodge spectra on Rips and Dowker--Rips complexes.
\end{abstract}

\maketitle

\section{Introduction}
\label{sec:intro}

Topological data analysis (TDA) studies the shape of data by replacing point clouds, metric spaces,
or relations with filtered simplicial complexes \cite{edelsbrunner2010computational}. Vietoris--Rips
complexes and Dowker complexes \cite{dowker1952homology} are standard vehicles for this passage
from data to topology, and Dowker duality has persistent and functorial formulations
\cite{chowdhury2018functorial}. The Dowker--Rips complex of Huber and Schnider
\cite{huber2025flagifying} is obtained by flagifying the Dowker complex, making it computationally
close to the Vietoris--Rips pipeline. These complexes encode topology through combinatorial data,
but in doing so they discard much of the underlying metric or witness geometry: once a simplex is
present, ordinary (co)homological computations treat it identically regardless of its size, shape,
or support in the data.

The classical object that couples geometry to (co)homology is the Hodge star. On a closed oriented
Riemannian manifold $(M,g)$, it is the metric-dependent operator
\[
\star : \Omega^k(M) \to \Omega^{n-k}(M),
\qquad
\alpha\wedge\star\beta = \langle\alpha,\beta\rangle_g\,\mathrm{vol}_g,
\]
which defines the codifferential and the Hodge--de~Rham Laplacian
$\Delta=d\delta+\delta d$ \cite{warner1983foundations,rosenberg1997laplacian}. The Hodge theorem
identifies harmonic forms with cohomology, so the exterior derivative carries topology while the
Hodge star carries metric geometry. Discrete Hodge Laplacians on graphs and simplicial complexes
inherit this separation \cite{lim2020hodge,horak2013spectra}. Related spectral tools include
normalized Hodge Laplacians and simplicial random walks \cite{schaub2018random}, Hodge-Laplacian
filters \cite{yang2022simplicialfilters,yang2023convolutional}, and weighted simplicial models
\cite{battiloro2023weighted}. Learning architectures also use Hodge-Laplacian structure
\cite{smirnov2021hodgenet,ebli2020simplicial}.

In discrete exterior calculus (DEC), the metric data is represented by a positive cochain inner
product, often a diagonal Hodge star
\[
W_k=\mathrm{diag}\{w_k(\sigma):\sigma\in K_k\},
\qquad
\langle\alpha,\beta\rangle_{k,W}=\alpha^TW_k\beta.
\]
For circumcentric meshes or alpha complexes, the canonical diagonal entries are primal--dual volume
ratios $(\star_k)_{\sigma\sigma}=|\star\sigma|/|\sigma|$
\cite{hirani2003dec,desbrun2005dec}. Rips and Dowker--Rips complexes have no canonical dual cells,
so this formula cannot be used literally. We instead regard the discrete Hodge star as a cochain
metric and choose positive diagonal weights from the data itself. This differs from weighted
persistent homology, where simplex weights alter the filtration
\cite{ren2018weighted,meng2019weighted}; here the real chain complex and its Betti numbers are fixed,
while the nonzero spectrum records geometry. It also differs from persistent Laplacians for
simplicial pairs \cite{memoli2020persistent}: the present implementation computes per-scale weighted
Hodge spectra, in the spirit of persistent spectral descriptors \cite{wang2020persistent}, rather
than a two-parameter persistent-pair Laplacian.

The main contribution of this paper is a framework for placing geometry-induced diagonal Hodge stars
on Rips-type and Dowker--Rips complexes, where no canonical primal--dual cell structure is
available. We study two concrete choices, simplex-volume weights and soft witness weights, with the
latter designed to quantify higher-order witness support lost under flagification. We also use the
self-witness specialization $Y=X$, in which the data points themselves serve as witnesses, to combine
Rips topology with Dowker-style witness geometry. For these weighted cochain metrics we prove the
basic Hodge-theoretic properties needed for the construction, including positivity, weighted
self-adjointness, and preservation of Betti numbers under positive diagonal weights. Finally, we
describe an implementation that constructs signed boundary matrices, weighted Hodge Laplacians,
spectra, heat traces, and entropy descriptors for Rips-type complexes up to a prescribed maximal
simplex dimension.

\section{Background: Dowker and Dowker--Rips Complexes}
\label{sec:background}

\begin{definition}[Dowker complex and witness set]
Let $X$ and $Y$ be finite sets, and let $R\subseteq X\times Y$ be a relation. The Dowker complex on
$X$ relative to $Y$, denoted $\Dow_R(X,Y)$, is the simplicial complex whose vertices are elements
of $X$, and where a finite subset $\sigma=\{x_0,\dots,x_k\}\subseteq X$ is a simplex if there
exists a common witness $y\in Y$ such that $(x_i,y)\in R$ for all $i=0,\dots,k$. Equivalently, the
witness set
\[
W_R(\sigma)=\{y\in Y : (x,y)\in R \text{ for all }x\in\sigma\}
\]
satisfies $\sigma\in\Dow_R(X,Y)$ if and only if $W_R(\sigma)\neq\varnothing$. In the metric
setting, if $X,Y\subseteq (Z,d)$ and $\epsilon\geq0$, we write
$R_\epsilon=\{(x,y)\in X\times Y:d(x,y)\leq\epsilon\}$ and denote the associated Dowker complex by
$\Dow_\epsilon(X,Y)$.
\end{definition}

\begin{definition}[Dowker--Rips complex]
The Dowker--Rips complex, introduced by Huber and Schnider \cite{huber2025flagifying}, is defined
as the flagification of the Dowker complex. For a relation $R\subseteq X\times Y$,
$\DR_R(X,Y)=\Flag(\Dow_R(X,Y))$; in the metric case,
\[
\DR_\epsilon(X,Y)
=
\Flag(\Dow_\epsilon(X,Y)).
\]
Equivalently, $\DR_\epsilon(X,Y)$ is the maximal flag complex whose one-skeleton agrees with the
one-skeleton of $\Dow_\epsilon(X,Y)$. Thus, two vertices $x_i,x_j\in X$ form an edge if
\[
\exists y\in Y
\quad\text{such that}\quad
d(x_i,y)\leq \epsilon
\quad\text{and}\quad
d(x_j,y)\leq \epsilon.
\]
A higher-dimensional simplex $\sigma=\{x_0,\dots,x_k\}$ belongs to $\DR_\epsilon(X,Y)$ if every
pair $\{x_i,x_j\}\subseteq \sigma$ is an edge. The distinction is therefore between requiring one
common witness for all vertices ($\Dow_\epsilon$) and requiring only pairwise witnesses
($\DR_\epsilon$). In particular, a simplex may belong to $\DR_\epsilon(X,Y)$ even if its common
witness set is empty:
\[
\sigma\in \DR_\epsilon(X,Y)
\quad\text{but}\quad
W_{R_\epsilon}(\sigma)=\varnothing.
\]
\end{definition}

This loss of higher-order witness information is central for defining Hodge stars.

\begin{definition}[Flagification and \texorpdfstring{$k$}{k}-flagification]
A simplicial complex $K$ is flag if every clique in its one-skeleton spans a simplex. The
flagification of $K$ is
\[
\Flag(K)=\mathrm{CliqueComplex}(K^{(1)}),
\]
where $K^{(1)}$ is the one-skeleton of $K$. More generally, $k$-flagification fills a simplex
whenever all of its $(k-1)$-dimensional faces are already present; ordinary flagification
corresponds to $k=2$. This creates a tradeoff: smaller $k$ is cheaper to compute but fills more
aggressively, whereas larger $k$ retains more higher-order information at greater computational
cost. For Dowker--Rips, ordinary flagification gives a complex determined by the one-skeleton, but
the loss of higher-order witness sets may be significant for Hodge-theoretic information.
\end{definition}

\section{Weighted Hodge Laplacians and Geometry-Induced Stars}
\label{sec:laplacian}
\label{sec:constructions}

\begin{definition}[Weighted discrete Hodge Laplacian]
Let $K$ be a finite simplicial complex and let $C^k(K)$ denote the space of $k$-cochains. Equip
$C^k(K)$ with the weighted inner product
$\langle \alpha,\beta\rangle_{k,W}=\alpha^T W_k\beta$, where
$W_k=\mathrm{diag}\{w_k(\sigma):\sigma\in K_k\}$ is positive diagonal. If
$d_k:C^k(K)\to C^{k+1}(K)$ is the coboundary operator, its weighted adjoint is
$\delta_{k+1}^{W}=W_k^{-1}d_k^T W_{k+1}$. The weighted Hodge Laplacian
$\Delta_k^W=\delta_{k+1}^{W}d_k+d_{k-1}\delta_k^W$ is, in coordinates,
\[
\Delta_k^W
=
W_k^{-1}d_k^T W_{k+1}d_k
+
d_{k-1}W_{k-1}^{-1}d_{k-1}^T W_k.
\]
The matrices $d_k$ encode the combinatorial topology of $K$, while the matrices $W_k$ encode
metric or geometric information. The main design choice is therefore the construction of $W_k$.
\end{definition}

\begin{remark}[Chain representative used in the implementation]
The theoretical results of Section~\ref{sec:theory} are stated for the weighted cochain Laplacian
$\Delta_k^W$, which is self-adjoint with respect to $\langle\cdot,\cdot\rangle_{k,W}$ but not
symmetric in the standard basis. The implementation works instead with signed boundary matrices
$\partial_k:C_k(K)\to C_{k-1}(K)$ (so that $d_{k-1}=\partial_k^T$) and assembles the analogous
weighted chain Laplacian induced by weighted inner products on chain spaces:
\[
L_k^W
=
W_k^{-1}\partial_k^T W_{k-1}\partial_k
+
\partial_{k+1}W_{k+1}^{-1}\partial_{k+1}^T W_k,
\]
which satisfies the same positivity and Betti-number preservation properties. For numerical
diagonalization, the implementation uses the symmetric representative
\[
\widehat{\Delta}_k^W
=
W_k^{-1/2}\partial_k^T W_{k-1}\partial_k W_k^{-1/2}
+
W_k^{1/2}\partial_{k+1}W_{k+1}^{-1}\partial_{k+1}^T W_k^{1/2}
=
W_k^{1/2}\,L_k^W\,W_k^{-1/2}.
\]
Because $\widehat{\Delta}_k^W$ is a similarity transform of $L_k^W$ by the positive diagonal
$W_k^{1/2}$, it is a real symmetric positive semidefinite matrix with the same spectrum as
$L_k^W$; in particular its eigenvalues, kernel dimension, and all derived spectral descriptors
coincide with those of the chain operator. The cochain form $\Delta_k^W$ and the chain form $L_k^W$
are related by the chain--cochain identification and by the placement of weights on neighbouring
degrees. All spectra reported below are computed from $\widehat{\Delta}_k^W$.
\end{remark}

For an alpha complex, the Hodge star can be defined using primal and dual volumes,
\[
(\star_k)_{\sigma\sigma}
=
\frac{|\star\sigma|}{|\sigma|}.
\]
For a Rips or Dowker--Rips complex, there is generally no canonical dual cell $\star\sigma$. We
therefore use the following diagonal substitute.

\begin{definition}[Geometry-induced diagonal Hodge star]
A geometry-induced diagonal Hodge star is a positive diagonal mass matrix
$W_k=\mathrm{diag}\{w_k(\sigma)\}$ whose entries are written
\[
(\star_k^W)_{\sigma\sigma}
=
w_k(\sigma),
\]
where $w_k(\sigma)>0$ is chosen using information available from the filtration, the metric, the
witness relation, or the local combinatorics. As in Section~\ref{sec:intro}, this is a weighted
inner product on $C^k(K)$, not a dimension-shifting map.
\end{definition}

The weighted Laplacian construction works for any positive diagonal matrices $W_k$. The substantive
choice is therefore the design of $w_k(\sigma)$. We use the identity star as a baseline, a
birth-scale star as a filtration baseline, and two geometry-aware stars that are implemented and
emphasized in the computations: simplex-volume weights and soft witness weights. Among these, the
soft witness star is the main Dowker--Rips construction, because it assigns a graded support value
to simplices created by flagification even when they have no exact common witness.

\begin{remark}[Cochain metrics beyond dual cells]
The constructions below should be read as choices of cochain metric rather than as attempts to
reconstruct a unique missing DEC dual mesh. In a finite-dimensional cochain space, any symmetric
positive definite matrix $M_k$ defines an inner product
\[
\langle\alpha,\beta\rangle_{k,M}
=
\alpha^T M_k\beta,
\]
and hence an adjoint $d_k^{*,M}=M_k^{-1}d_k^T M_{k+1}$ and a Hodge Laplacian. A diagonal DEC
Hodge star on a circumcentric mesh is one special case, where
$M_k(\sigma,\sigma)=|\star\sigma|/|\sigma|$. For Rips-type complexes we instead take
$M_k=W_k$ diagonal and design the entries from available data: filtration scale, Euclidean simplex
volume, or witness support. This keeps the Hodge-theoretic principle---geometry enters through the
cochain metric---without imposing a primal--dual cell interpretation where none is canonical.
\end{remark}

\begin{definition}[Identity Hodge star]
The simplest choice is $w_k(\sigma)=1$, i.e.\ $W_k=I$. This yields the standard combinatorial
Hodge Laplacian; it provides a baseline but ignores metric and witness information.
\end{definition}

\begin{definition}[Birth-scale Hodge star]
For a Rips simplex $\sigma$, define its birth value by
$b(\{x\})=0$ for vertices and, for $|\sigma|\geq2$,
$b(\sigma)=\max_{x_i,x_j\in\sigma} d(x_i,x_j)$, the filtration value at which $\sigma$ enters the
Vietoris--Rips filtration. A birth-scale Hodge star is
$w_k^{\mathrm{birth}}(\sigma)=\phi(b(\sigma))$, where $\phi:[0,\infty)\to(0,\infty)$ is positive,
for example $\phi(b)=b+\eta$, $\phi(b)=1/(b+\eta)$, or
$\phi(b)=\exp(-b^2/t)$. This is filtration-compatible because it uses the same birth information
that defines the Rips filtration.
\end{definition}

\begin{definition}[Simplex-volume Hodge star]
If the point cloud lies in Euclidean space, each simplex can be assigned a geometric volume. Here
$|\sigma|_k$ denotes the $k$-dimensional Euclidean volume of $\sigma$, with $|\sigma|_0=1$ for
vertices; the implementation uses edge length for $1$-simplices, Heron's formula for triangles, and
the Cayley--Menger determinant in higher dimension. Then
\[
w_k^{\mathrm{vol}}(\sigma)
=
|\sigma|_k+\eta
\qquad\text{or}\qquad
w_k^{\mathrm{invvol}}(\sigma)
=
\frac{1}{|\sigma|_k+\eta}.
\]
This injects geometric information into the Rips complex; nearly degenerate simplices have very
small volume, so regularization is necessary.
\end{definition}

\begin{definition}[Witness-supported Hodge star]
For Dowker and Dowker--Rips complexes, a natural replacement for a dual object is the witness set
$W_R(\sigma)$. When an ambient Euclidean geometry is available, a direct witness-supported Hodge
star is
\[
w_k^{\mathrm{wit}}(\sigma)
=
\frac{\mu(W_R(\sigma))+\eta}{|\sigma|_k+\delta},
\]
where $\mu$ is a measure or counting measure on $Y$ and $\eta,\delta>0$ are regularization
parameters. This mimics the DEC formula $(\star_k)_{\sigma\sigma}=|\star\sigma|/|\sigma|$, with
$\mu(W_R(\sigma))$ playing the role of the dual volume. For a purely relational Dowker complex, the
denominator can be replaced by $1$ or by another prescribed positive simplex-size function. For
ordinary Dowker complexes every simplex has a nonempty witness set, but flagification can introduce
simplices with empty witness sets, motivating a soft witness construction.
\end{definition}

\begin{definition}[Soft witness Hodge star]
Suppose $X,Y\subseteq (Z,d)$. Define the soft witness support and the corresponding weight by
\[
s_t(\sigma;Y)=\sum_{y\in Y}\exp\!\left(-\frac{1}{t}\sum_{x\in\sigma} d(x,y)^2\right),
\qquad
w_k^{\mathrm{soft}}(\sigma)=\frac{s_t(\sigma;Y)+\eta}{|\sigma|_k+\delta}.
\]
This is always positive and measures how strongly $\sigma$ is supported by the witness space $Y$.
A simplex inserted only by flagification, with no true common witness, may still receive a small
but nonzero weight. This is the central construction for Hodge theory on Dowker--Rips complexes:
Dowker--Rips gives efficient topology, while soft witness weights quantify higher-order witness
support lost under flagification. The denominator uses the Euclidean simplex volume, with
regularization, in the current implementation.
\end{definition}

\begin{remark}[Implemented specializations]
The implementation accompanying this paper computes two geometry-aware Hodge stars explicitly:
the simplex-volume Hodge star and the soft witness Hodge star. For the volume star, simplex
volumes are computed from the point coordinates using edge lengths, triangle areas, and the
Cayley--Menger determinant in higher dimension, with the regularized weights
\[
w_k^{\mathrm{vol}}(\sigma)=|\sigma|_k+\eta
\qquad\text{or}\qquad
w_k^{\mathrm{invvol}}(\sigma)=\frac{1}{|\sigma|_k+\eta}.
\]
For the soft witness star on a Rips complex, the default computational choice is the
self-witness set $Y=X$, so that
\[
s_t^{X}(\sigma)
=
\sum_{x_j\in X}
\exp\!\left(
-\frac{1}{t}
\sum_{x_i\in\sigma}d(x_i,x_j)^2
\right),
\qquad
w_k^{\mathrm{soft}}(\sigma)
=
\frac{s_t^{X}(\sigma)+\eta}{|\sigma|_k+\delta}.
\]
Thus the current computations compare ordinary combinatorial spectra with spectra obtained from
simplex-volume weights and self-witness soft Dowker weights. The formulas are not restricted to
planar data: the point cloud may lie in an ambient Euclidean space of arbitrary coordinate
dimension, and the implemented Rips construction includes tetrahedra when the maximal simplex
dimension is set to $3$.
\end{remark}

\section{Theoretical Properties}
\label{sec:theory}

Throughout, self-adjointness is taken with respect to the weighted inner product. Thus an operator
$T:C^k(K)\to C^k(K)$ is self-adjoint if
\[
\langle T\alpha,\beta\rangle_{k,W}
=
\langle\alpha,T\beta\rangle_{k,W}
\qquad
\text{for all }\alpha,\beta\in C^k(K),
\]
equivalently if $W_kT$ is symmetric. Note that $\Delta_k^W$ is generally not symmetric in the
standard basis, so this qualifier is essential.

\begin{proposition}[Asymptotic soft-witness decay rate]
\label{prop:witness-separation}
Let $X,Y$ be finite metric spaces in a common ambient metric space. Define
$s_t(\sigma;Y)=\sum_{y\in Y}\exp(-t^{-1}\sum_{x\in\sigma}d(x,y)^2)$, set
$D(\sigma,y)=\sum_{x\in\sigma}d(x,y)^2$, and let
$D^{*}(\sigma)=\min_{y\in Y}D(\sigma,y)$.
Then
\[
\lim_{t\to0^{+}}\bigl(-t\log s_t(\sigma;Y)\bigr)=D^{*}(\sigma).
\]
Moreover, for the metric relation $R_\epsilon=\{(x,y):d(x,y)\le\epsilon\}$:
\begin{enumerate}[leftmargin=2em,label=(\roman*)]
\item if $\sigma\in\Dow_\epsilon(X,Y)$, then $D^{*}(\sigma)\le \#\sigma\,\epsilon^2$;
\item if $\sigma\in\DR_\epsilon(X,Y)\setminus\Dow_\epsilon(X,Y)$, then
$D^{*}(\sigma)>\epsilon^2$.
\end{enumerate}
Thus the soft witness support has a precise small-bandwidth decay rate, and flagified-only
simplices have decay rate bounded below by $\epsilon^2$.
\end{proposition}

\begin{proof}
Since $Y$ is finite, the largest exponential term controls the sum:
\[
e^{-D^{*}(\sigma)/t}\le s_t(\sigma;Y)\le |Y|e^{-D^{*}(\sigma)/t}.
\]
Taking $-t\log(\cdot)$ gives
\[
D^{*}(\sigma)-t\log |Y|\le -t\log s_t(\sigma;Y)\le D^{*}(\sigma),
\]
and the limit follows. If $\sigma\in\Dow_\epsilon(X,Y)$, there exists a common witness $y_0$
with $d(x,y_0)\le\epsilon$ for every $x\in\sigma$, so
$D^{*}(\sigma)\le D(\sigma,y_0)\le \#\sigma\,\epsilon^2$. If
$\sigma\in\DR_\epsilon(X,Y)\setminus\Dow_\epsilon(X,Y)$, then no $y\in Y$ witnesses all vertices.
For each $y$, some $x\in\sigma$ satisfies $d(x,y)>\epsilon$, hence
$D(\sigma,y)>\epsilon^2$. Taking the minimum over finite $Y$ gives
$D^{*}(\sigma)>\epsilon^2$.
\end{proof}

\begin{remark}
Proposition~\ref{prop:witness-separation} is an asymptotic statement, not a pointwise guarantee
that every flagified-only simplex has smaller $s_t$ than every true Dowker simplex at a fixed
$t$. The rate ranges can overlap for higher-dimensional simplices. The proposition nevertheless
gives a concrete quantity to test experimentally: whether the decay-rate distinction changes the
nonzero spectrum of the weighted Hodge Laplacian.
\end{remark}

\begin{lemma}[Weighted adjoint]
\label{lem:adjoint}
With respect to the weighted inner products $\langle\alpha,\beta\rangle_{k,W}=\alpha^TW_k\beta$, the
adjoint of the coboundary $d_k:C^k(K)\to C^{k+1}(K)$ is $d_k^{*}=\delta_{k+1}^{W}=W_k^{-1}d_k^{T}W_{k+1}$.
\end{lemma}

\begin{proof}
For all $\alpha\in C^k(K)$, $\beta\in C^{k+1}(K)$,
\[
\langle d_k\alpha,\beta\rangle_{k+1,W}
=(d_k\alpha)^T W_{k+1}\beta
=\alpha^T W_k\bigl(W_k^{-1}d_k^T W_{k+1}\bigr)\beta
=\langle\alpha,\delta_{k+1}^{W}\beta\rangle_{k,W}.
\]
As $W_k$ is invertible, $\delta_{k+1}^{W}$ is the unique adjoint of $d_k$.
\end{proof}

\begin{proposition}[Positivity and self-adjointness]
\label{prop:selfadjoint}
Let $W_{k-1},W_k,W_{k+1}$ be positive definite diagonal matrices. Then the weighted Hodge Laplacian
$\Delta_k^W=\delta_{k+1}^{W}d_k+d_{k-1}\delta_k^{W}$ is self-adjoint and positive semidefinite with
respect to $\langle\cdot,\cdot\rangle_{k,W}$.
\end{proposition}

\begin{proof}
By Lemma~\ref{lem:adjoint}, $\delta_{k+1}^{W}=d_k^{*}$ and $\delta_k^{W}=d_{k-1}^{*}$, so
$\Delta_k^W=d_k^{*}d_k+d_{k-1}d_{k-1}^{*}$. For any linear map $A$ with adjoint $A^{*}$, the operator
$A^{*}A$ is self-adjoint, since $\langle A^{*}A\alpha,\beta\rangle=\langle A\alpha,A\beta\rangle
=\langle\alpha,A^{*}A\beta\rangle$, and positive semidefinite, since
$\langle A^{*}A\alpha,\alpha\rangle=\|A\alpha\|_W^2\ge0$. Applying this to $A=d_k$ and $A=d_{k-1}^{*}$
shows each summand, and hence their sum, is self-adjoint and positive semidefinite.
\end{proof}

\begin{proposition}[Kernel dimension and Betti-number preservation]
\label{prop:homology}
Let $W_{k-1},W_k,W_{k+1}$ be positive definite diagonal matrices. Then
\[
\dim\ker\Delta_k^W=\dim H^k(K;\R)=\beta_k(K),
\]
independently of the choice of positive weights: the weights do not change the kernel dimension,
only the nonzero spectrum.
\end{proposition}

\begin{proof}
Adjoints are taken in the weighted inner products (Lemma~\ref{lem:adjoint}). Since $\Delta_k^W
=d_k^{*}d_k+d_{k-1}d_{k-1}^{*}$ is positive semidefinite (Proposition~\ref{prop:selfadjoint}),
\[
\begin{aligned}
\Delta_k^W\alpha=0
&\iff
\langle\Delta_k^W\alpha,\alpha\rangle_{k,W}=0\\
&\iff
\|d_k\alpha\|_{k+1,W}^2+\|\delta_k^{W}\alpha\|_{k-1,W}^2=0\\
&\iff
d_k\alpha=0\ \text{and}\ \delta_k^{W}\alpha=0,
\end{aligned}
\]
so $\ker\Delta_k^W=\ker d_k\cap(\operatorname{im}d_{k-1})^{\perp_W}$ is the space of harmonic
cochains. The weighted orthogonal (Hodge) decomposition
\[
C^k(K)=\operatorname{im}d_{k-1}\oplus\ker\Delta_k^W\oplus\operatorname{im}\delta_{k+1}^{W}
\]
then gives $\ker\Delta_k^W\cong\ker d_k/\operatorname{im}d_{k-1}=H^k(K;\R)$. Over the field
$\R$, this dimension equals the homological Betti number $\dim H_k(K;\R)$. Thus the kernel
dimension is independent of the positive weights.
\end{proof}

\begin{remark}
Proposition~\ref{prop:homology} is the weighted discrete Hodge theorem: harmonic cochains represent
cohomology for any inner product, so all positive Hodge stars preserve Betti numbers. The geometric
content of the weights lives entirely in the nonzero part of the spectrum.
\end{remark}

\section{Spectral Descriptors and Implementation}
\label{sec:descriptors}

For each filtration scale $\epsilon$, let $K_\epsilon$ be the corresponding Rips-type complex and
construct a weighted discrete Hodge Laplacian $\Delta_k^W(\epsilon)$ on $k$-cochains. Its spectrum
is written
\[
0=\lambda_1^{(k)}(\epsilon)\leq \lambda_2^{(k)}(\epsilon)\leq \cdots.
\]
By Proposition~\ref{prop:homology}, the kernel dimension
\[
\beta_k(\epsilon)=\dim\ker \Delta_k^W(\epsilon)
\]
is independent of the choice of positive diagonal weights and agrees with the ordinary Betti number
of $K_\epsilon$. The nonzero eigenvalues, however, depend on the chosen Hodge star and provide
geometry-dependent information at that scale, including the first nonzero eigenvalues
$\lambda_{\beta_k+1}^{(k)}(\epsilon),\dots$ and spectral gaps
$\lambda_{j+1}^{(k)}(\epsilon)-\lambda_j^{(k)}(\epsilon)$. We also record heat-type summaries such
as the heat trace
$\mathrm{Tr}(e^{-s\Delta_k^W(\epsilon)})=\sum_j e^{-s\lambda_j^{(k)}(\epsilon)}$ and the spectral
entropy $-\sum_j p_j\log p_j$, where
$p_j=e^{-s\lambda_j}/\sum_\ell e^{-s\lambda_\ell}$.

Across the filtration, the per-scale eigenvalues and heat summaries define persistent spectral
curves
$\epsilon\mapsto\lambda_j^{(k)}(\epsilon)$ and
$\epsilon\mapsto\mathrm{Tr}(e^{-s\Delta_k^W(\epsilon)})$, and similarly for spectral entropy.
These curves may distinguish datasets whose persistent homology is similar but whose metric or
witness geometry differs.

The terminology ``persistent spectral curve'' is used here in a deliberately limited sense. The
weights $W_k(\epsilon)$ do not alter the simplicial filtration $\{K_\epsilon\}$ and therefore do
not change the ordinary persistent homology module. The inclusion maps
$K_{\epsilon_i}\subseteq K_{\epsilon_j}$ and the induced maps
$H_k(K_{\epsilon_i};\mathbb R)\to H_k(K_{\epsilon_j};\mathbb R)$ are independent of the positive
diagonal weights. What changes with the choice of Hodge star is the nonzero spectrum of the
per-scale operator $\Delta_k^W(\epsilon)$. Thus the proposed descriptors should be viewed as
geometry-aware spectral summaries along a filtration, complementary to persistent homology, rather
than as a persistent-pair Laplacian theory.

This differs from weighted persistent homology, where weights modify the filtration or persistence
computation itself. It also differs from persistent Laplacians for simplicial pairs, which use a
pair $K_{\epsilon_i}\subseteq K_{\epsilon_j}$ and whose kernel dimension equals the corresponding
persistent Betti number. Our operators are single-scale weighted Hodge Laplacians on $K_\epsilon$:
their kernels give ordinary Betti numbers at that scale, while their nonzero spectra provide
additional geometry-aware features.

The accompanying implementation follows this pipeline. It constructs signed real boundary matrices
from the ordered simplices, assembles the symmetric representative of the weighted Hodge Laplacian,
and computes spectra by dense self-adjoint eigensolvers for the small complexes considered here.
The implemented Hodge-star choices are $W_k=I$, birth-scale weights, simplex-volume and
inverse-volume weights, and self-witness soft Dowker weights. The validation tests check
$\partial_k\partial_{k+1}=0$, symmetry and positive semidefiniteness of the Laplacians, and
$\dim\ker\Delta_k^W=\beta_k$ for all positive Hodge stars. Additional tests compare spectral
descriptors for point clouds with identical $H_0$ persistence but different $L_0$ spectra, and for
identity, volume, and soft-witness stars on the same Rips complex. The implementation also exercises
three-dimensional input and tetrahedral complexes: with four points in $\mathbb{R}^3$ and
maximal simplex dimension $3$, the Rips complex contains all $15$ simplices of a tetrahedron, and
the code computes both $L_2$ spectra, where tetrahedra contribute through the upper boundary
operator, and the top-dimensional $L_3$ spectrum.

\begin{remark}[Parameters and cost]
The regularizers $\eta,\delta>0$ keep the diagonal weights positive and prevent degenerate
simplices from producing singular volume factors. The bandwidth $t$ controls how local the soft
witness support is: small $t$ emphasizes the best witnesses through the decay rate in
Proposition~\ref{prop:witness-separation}, while larger $t$ averages over more of $Y$. For a
complex $K$ with witness set $Y$, the soft witness weights cost
$O(|K|\,|Y|\,d)$ in an ambient $d$-dimensional Euclidean space. The current implementation uses
dense eigensolvers, with cubic cost in the number of $k$-simplices, which is appropriate for the
small examples below but not for large-scale computation.
\end{remark}

\section{Experiments}
\label{sec:experiments}

We report four validation examples, each tied to one of the claims above. The experiments are
deliberately small and reproducible: they test that positive diagonal Hodge-star weights preserve
Betti numbers, that changing the cochain metric changes the nonzero spectrum, and that soft witness
weights record support information lost under flagification. Persistent homology is used only as a
baseline comparison in the first example. As in Section~\ref{sec:descriptors}, the persistent
objects computed here are per-scale spectral descriptors, not two-parameter persistent-pair
Laplacians.

\begin{example}[Persistent baseline: same barcode, different spectrum]
The first example compares two four-point clouds whose minimum-spanning-tree edge lengths are all
equal to one. Their $H_0$ persistence barcodes are therefore identical: three finite bars die at
scale $1$ and one bar persists. Explicitly,
\[
\begin{aligned}
X_{\mathrm{path}}&=\{(0,0),(1,0),(2,0),(3,0)\},\\
X_{\mathrm{star}}&=\{(0,0),(1,0),(-\tfrac12,\tfrac{\sqrt3}{2}),
(-\tfrac12,-\tfrac{\sqrt3}{2})\}.
\end{aligned}
\]
\begin{center}
\begin{tikzpicture}[
    scale=0.85,
    point/.style={circle,fill=black,inner sep=1.8pt},
    edge/.style={line width=0.8pt},
    every node/.style={font=\small}
]
    \node at (1.5,0.75) {$X_{\mathrm{path}}$};
    \draw[edge] (0,0)--(1,0)--(2,0)--(3,0);
    \foreach \x in {0,1,2,3} {
        \node[point] at (\x,0) {};
    }

    \begin{scope}[xshift=6cm]
        \node at (0,1.25) {$X_{\mathrm{star}}$};
        \coordinate (c) at (0,0);
        \coordinate (a) at (1,0);
        \coordinate (b) at (-0.5,0.866);
        \coordinate (d) at (-0.5,-0.866);
        \draw[edge] (c)--(a) (c)--(b) (c)--(d);
        \foreach \p in {c,a,b,d} {
            \node[point] at (\p) {};
        }
    \end{scope}
\end{tikzpicture}
\end{center}
At the Rips scale $\epsilon=1.05$ with maximal simplex dimension $2$, each complex has $7$ simplices
($4$ vertices, $3$ edges, no triangle), so $\beta_0=1$ and $\beta_1=0$; one graph is a path and the
other is a three-legged star. Using the identity star $W_k=I$ at dimension $k=0$, the $L_0$ spectra
distinguish them:
\[
\mathrm{Spec}(L_0^{\mathrm{path}})
=
\{0,0.586,2,3.414\},
\qquad
\mathrm{Spec}(L_0^{\mathrm{star}})
=
\{0,1,1,4\}.
\]
Thus the same $H_0$ persistence barcode does not imply the same Hodge spectrum. This is the
baseline motivation for adding spectral descriptors to ordinary persistent homology.
\end{example}

\begin{example}[Fixed complex, different Hodge stars]
The second example isolates the effect of the Hodge star by fixing both the data and the simplicial
complex. Let
$X=\{(1.5\cos\theta_j,\,0.7\sin\theta_j):\theta_j=2\pi j/6,\ j=0,\dots,5\}$, and take the Rips scale
$\epsilon=1.7$ with maximal simplex dimension $2$. The complex has $16$ simplices ($6$ vertices,
$8$ edges, $2$ triangles) with $\beta_0=1$ and $\beta_1=1$.
\begin{center}
\begin{tikzpicture}[
    scale=1.25,
    point/.style={circle,fill=black,inner sep=1.7pt},
    edge/.style={line width=0.75pt},
    face/.style={fill=gray!15,draw=none},
    every node/.style={font=\small}
]
    \coordinate (v0) at (1.5,0);
    \coordinate (v1) at (0.75,0.606);
    \coordinate (v2) at (-0.75,0.606);
    \coordinate (v3) at (-1.5,0);
    \coordinate (v4) at (-0.75,-0.606);
    \coordinate (v5) at (0.75,-0.606);

    \fill[face] (v0)--(v1)--(v5)--cycle;
    \fill[face] (v2)--(v3)--(v4)--cycle;
    \draw[edge] (v0)--(v1)--(v2)--(v3)--(v4)--(v5)--(v0);
    \draw[edge] (v1)--(v5);
    \draw[edge] (v2)--(v4);

    \foreach \p in {v0,v1,v2,v3,v4,v5} {
        \node[point] at (\p) {};
    }
    \node[anchor=west] at (v0) {$x_0$};
    \node[anchor=south west] at (v1) {$x_1$};
    \node[anchor=south east] at (v2) {$x_2$};
    \node[anchor=east] at (v3) {$x_3$};
    \node[anchor=north east] at (v4) {$x_4$};
    \node[anchor=north west] at (v5) {$x_5$};
\end{tikzpicture}
\end{center}
We compute spectra for the same complex
using the identity star, the simplex-volume star, and the self-witness soft star (bandwidth
$t=1.0$, regularizers $\eta=10^{-8}$, $\delta=10^{-6}$). The kernel dimension agrees across all
positive stars, as predicted by Proposition~\ref{prop:homology}, while the nonzero eigenvalues move;
the corresponding $L_0$ and $L_1$ scatter plots appear in Figure~\ref{fig:weighted-stars}. The
$L_0$ spectra are
\[
\begin{array}{c|c}
\text{star} & L_0\text{ spectrum}\\
\hline
\text{identity} & 0,\ 1,\ 3,\ 3,\ 4,\ 5\\
\text{volume} & 0,\ 0.748,\ 2.687,\ 3.111,\ 3.696,\ 4.020\\
\text{soft witness} & 0,\ 3.590,\ 5.865,\ 8.797,\ 25.680,\ 32.201
\end{array}
\]
The first
nonzero $L_0$ eigenvalue moves from $1.00$
(identity) to $0.748$ (volume) to $3.59$ (soft witness). This is the basic empirical behavior needed
for geometry-aware descriptors: topology is preserved in the kernel, and geometry appears in the
positive spectrum.
\end{example}

\begin{figure}[t]
\centering
\includegraphics[width=0.88\textwidth]{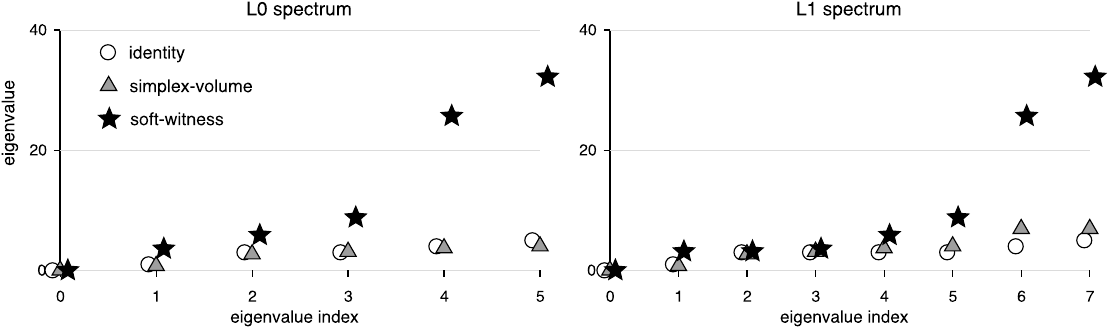}
\caption{Scatter plots of the $L_0$ and $L_1$ spectra for one Rips complex with identity,
simplex-volume, and soft-witness stars. The Betti numbers are unchanged, but the nonzero spectra
are reshaped by the choice of Hodge star.}
\label{fig:weighted-stars}
\end{figure}

\begin{example}[Soft witness support under flagification]
The third example tests the soft witness construction from Definition~3.9 and the decay-rate
interpretation in Proposition~\ref{prop:witness-separation}. Use the relation
\[
\begin{pmatrix}
1&0&1\\
1&1&0\\
0&1&1
\end{pmatrix}.
\]
The ordinary Dowker complex contains the three edges but no triangle, while the Dowker--Rips
flagification inserts the triangle. This relation is realized metrically by taking the landmarks to
be the vertices of a unit equilateral triangle,
$X=\{(0,0),(1,0),(\tfrac12,\tfrac{\sqrt3}{2})\}$, and the witnesses to be the three edge midpoints,
$Y=\{(\tfrac12,0),(\tfrac34,\tfrac{\sqrt3}{4}),(\tfrac14,\tfrac{\sqrt3}{4})\}$, under $R_\epsilon$
with $\epsilon=0.51$. All distances $d(x,y)$ are Euclidean, $D(\sigma,y)=\sum_{x\in\sigma}d(x,y)^2$,
and $D^*(\sigma)=\min_{y\in Y}D(\sigma,y)$; the denominator uses the Euclidean simplex size (edge
length or triangle area), with regularizers $\eta=\delta=10^{-8}$. The supported Dowker cycle and
the flagified simplex are shown below.
\begin{center}
\begin{tikzpicture}[
    scale=1.35,
    point/.style={circle,fill=black,inner sep=1.6pt},
    edge/.style={line width=0.75pt},
    support/.style={line width=0.45pt,gray!60},
    face/.style={fill=gray!14,draw=none},
    every node/.style={font=\small}
]
    \coordinate (a0) at (0,0);
    \coordinate (a1) at (1,0);
    \coordinate (a2) at (0.5,0.866);
    \coordinate (b0) at (2.8,0);
    \coordinate (b1) at (3.8,0);
    \coordinate (b2) at (3.3,0.866);

    \node at (0.5,1.24) {Dowker support};
    \draw[edge] (a0)--(a1)--(a2)--(a0);
    \draw[support] (0.5,0)--(a0) (0.5,0)--(a1);
    \draw[support] (0.75,0.433)--(a1) (0.75,0.433)--(a2);
    \draw[support] (0.25,0.433)--(a0) (0.25,0.433)--(a2);
    \foreach \p/\lab/\pos in {a0/$x_0$/below left,a1/$x_1$/below right,a2/$x_2$/above} {
        \node[point] at (\p) {};
        \node[\pos] at (\p) {\lab};
    }
    \foreach \x/\y/\lab/\pos in {0.5/0/$y_0$/below,0.75/0.433/$y_1$/right,0.25/0.433/$y_2$/left} {
        \node[inner sep=0pt] at (\x,\y) {$\star$};
        \node[\pos] at (\x,\y) {\lab};
    }

    \node at (3.3,1.24) {flagified simplex};
    \fill[face] (b0)--(b1)--(b2)--cycle;
    \draw[edge] (b0)--(b1)--(b2)--(b0);
    \foreach \p/\lab/\pos in {b0/$x_0$/below left,b1/$x_1$/below right,b2/$x_2$/above} {
        \node[point] at (\p) {};
        \node[\pos] at (\p) {\lab};
    }
\end{tikzpicture}
\end{center}

\begin{center}
\normalfont
\refstepcounter{table}\label{tab:witness-separation}
\begin{tabular}{lccccc}
\hline
simplex & common witnesses & $D^*$ & $-t\log s_t$ & $s_t$ & $w^{\mathrm{soft}}$\\
\hline
$\{0,1\}$ & $1$ & $0.5$ & $0.470$ & $9.56\cdot 10^{-2}$ & $9.56\cdot 10^{-2}$\\
$\{0,1,2\}$ & $0$ & $1.25$ & $1.030$ & $5.79\cdot 10^{-3}$ & $1.34\cdot 10^{-2}$\\
\hline
\end{tabular}
\par\vspace{0.5em}
\textsc{Table~\thetable.} Numerical witness separation for a true Dowker edge and a
flagified-only triangle.
\end{center}
This yields $D^*(\{0,1\})=0.5$ (attained at the midpoint of edge $\{0,1\}$) and
$D^*(\{0,1,2\})=1.25$. With
$t=0.2$, the true Dowker edge $\{0,1\}$ and the flagified-only triangle $\{0,1,2\}$ have the values
in Table~\ref{tab:witness-separation}. The triangle has no common witness and has larger $D^*$,
smaller soft support, and much smaller witness weight. This is the intended numerical role of
$s_t(\sigma;Y)$: it assigns graded support to flagified simplices without treating them as genuinely
witnessed Dowker simplices.
\end{example}

\begin{remark}[Spectral effect of witness separation]
The same example also shows the spectral effect of changing from the ordinary Dowker complex to its
flagified Dowker--Rips complex. For the Dowker complex, the binary witness star gives eigenvalues
approximately $\{0,6,6\}$, so $\beta_1=1$. After flagification, the inserted triangle kills this
class and the binary witness spectrum becomes approximately
$\{6,6,1.30\cdot10^8\}$; the large eigenvalue reflects the tiny regularized witness weight of the
unsupported triangle. The soft star gives a less singular flagified spectrum,
approximately $\{18.73,18.73,21.43\}$.
\end{remark}

\begin{example}[Tetrahedral three-dimensional validation]
The fourth example checks that the implemented construction is not limited to planar point clouds or
two-dimensional complexes. Take the four points
\[
X=\{(0,0,0),(1,0,0),(0,1,0),(0,0,1)\}\subset\mathbb{R}^3
\]
and build the Vietoris--Rips complex at $\epsilon=1.5$ with maximal simplex dimension $3$. This
includes all $15$ simplices of the tetrahedron: $4$ vertices, $6$ edges, $4$ triangular faces, and
one tetrahedral $3$-simplex. The simplex is shown below.
\begin{center}
\begin{tikzpicture}[
    scale=1.15,
    point/.style={circle,fill=black,inner sep=1.7pt},
    edge/.style={line width=0.75pt},
    hidden/.style={line width=0.55pt,dashed,gray!65},
    face/.style={fill=gray!13,draw=none},
    every node/.style={font=\small}
]
    \coordinate (x0) at (-1.0,-0.55);
    \coordinate (x1) at (1.15,-0.55);
    \coordinate (x2) at (0.35,1.0);
    \coordinate (x3) at (0.05,0.05);

    \fill[face] (x0)--(x1)--(x2)--cycle;
    \fill[gray!20,opacity=0.45] (x0)--(x1)--(x3)--cycle;
    \fill[gray!22,opacity=0.38] (x1)--(x2)--(x3)--cycle;
    \fill[gray!16,opacity=0.38] (x2)--(x0)--(x3)--cycle;

    \draw[edge] (x0)--(x1)--(x2)--(x0);
    \draw[edge] (x0)--(x3) (x1)--(x3) (x2)--(x3);

    \foreach \p/\lab/\pos in {x0/$x_0$/below left,x1/$x_1$/below right,x2/$x_2$/above,x3/$x_3$/right} {
        \node[point] at (\p) {};
        \node[\pos] at (\p) {\lab};
    }
\end{tikzpicture}
\end{center}

The command-line pipeline computes the simplex-volume and soft-witness Hodge spectra on
$2$-cochains, so the tetrahedron contributes through the upper boundary term
\[
L_2^W
=
W_2^{-1}B_2^TW_1B_2 + B_3W_3^{-1}B_3^TW_2,
\]
and it also computes the top-dimensional spectrum on $3$-cochains, where only the down term is
present.

For the simplex-volume star, the $L_2$ eigenvalues are approximately
\[
\{7.727,\ 8.828,\ 8.828,\ 14.196\},
\]
with $\beta_2=0$. The self-witness soft star gives
\[
\{3.740,\ 5.806,\ 5.806,\ 8.064\},
\]
again with $\beta_2=0$. The volume-weighted $L_3$ spectrum consists of the single eigenvalue
$14.196$, with $\beta_3=0$. This verifies the intended three-dimensional path: ambient
$\mathbb{R}^3$ input, tetrahedral Rips construction, Cayley--Menger volume weights, soft witness
weights, and weighted Hodge spectra in dimensions $2$ and $3$.
\end{example}

\section{Conclusion}
\label{sec:conclusion}

Flagification makes Dowker and Rips-type complexes computationally efficient, but it may destroy
higher-order witness or metric information. For Hodge theory, one should therefore ask not only
which simplices are present, but also what cochain metric should measure their geometric or witness
support. The DEC volume-ratio formula suggests one useful design pattern,
\[
(\star_k)_{\sigma\sigma}
=
\frac{\text{surrogate dual support of }\sigma}{\text{primal size of }\sigma},
\]
but this quotient should be understood as one interpretable cochain-metric design rather than as a
requirement that a literal dual cell exist. For Dowker--Rips complexes, the support term can be
supplied by witness sets, soft witness measures, or birth scales. The simplex-volume and soft
witness constructions studied here make this principle concrete: the former injects Euclidean
simplex geometry into the Hodge star, while the latter reintroduces information about higher-order
witness support that may be lost under flagification. More general future choices could use sparse
non-diagonal mass matrices, kernel metrics, density-aware weights, or learned positive definite
cochain metrics. Another natural next step is to study how geometry-induced Hodge stars vary across
filtrations and whether the resulting spectral curves satisfy stability or interleaving-type
estimates. Such a theory would require compatibility of the weighted cochain metrics under
inclusions $K_{\epsilon_i}\subseteq K_{\epsilon_j}$ and is beyond the fixed-scale framework
developed here.

\section*{Acknowledgments}
This work was supported in part by NSF grant DMS-2514195 and by a University of Arkansas Honors
College research grant.

\providecommand{\bysame}{\leavevmode\hbox to3em{\hrulefill}\thinspace}
\providecommand{\MR}{\relax\ifhmode\unskip\space\fi MR }
\providecommand{\MRhref}[2]{%
  \href{http://www.ams.org/mathscinet-getitem?mr=#1}{#2}
}
\providecommand{\href}[2]{#2}


\begin{thebibliography}{10}

\bibitem{battiloro2023weighted}
Claudio Battiloro, Stefania Sardellitti, Sergio Barbarossa, and Paolo
  Di~Lorenzo, \emph{Topological signal processing over weighted simplicial
  complexes}, 2023, arXiv:2302.08561.

\bibitem{chowdhury2018functorial}
Samir Chowdhury and Facundo M{\'e}moli, \emph{A functorial {D}owker theorem and
  persistent homology of asymmetric networks}, Journal of Applied and
  Computational Topology \textbf{2} (2018), no.~1--2, 115--175.

\bibitem{desbrun2005dec}
Mathieu Desbrun, Anil~N. Hirani, Melvin Leok, and Jerrold~E. Marsden,
  \emph{Discrete exterior calculus}, 2005, arXiv:math/0508341.

\bibitem{dowker1952homology}
C.~H. Dowker, \emph{Homology groups of relations}, Annals of Mathematics
  \textbf{56} (1952), no.~1, 84--95.

\bibitem{ebli2020simplicial}
Stefania Ebli, Micha{\"e}l Defferrard, and Gard Spreemann, \emph{Simplicial
  neural networks}, 2020, arXiv:2010.03633.

\bibitem{edelsbrunner2010computational}
Herbert Edelsbrunner and John~L. Harer, \emph{Computational topology: An
  introduction}, American Mathematical Society, 2010.

\bibitem{hirani2003dec}
Anil~N. Hirani, \emph{Discrete exterior calculus}, Ph.D. thesis, California
  Institute of Technology, 2003.

\bibitem{horak2013spectra}
Danijela Horak and J{\"u}rgen Jost, \emph{Spectra of combinatorial {L}aplace
  operators on simplicial complexes}, Advances in Mathematics \textbf{244}
  (2013), 303--336.

\bibitem{huber2025flagifying}
Marius Huber and Patrick Schnider, \emph{Flagifying the {D}owker complex},
  2025, arXiv:2508.08025.

\bibitem{lim2020hodge}
Lek-Heng Lim, \emph{Hodge {L}aplacians on graphs}, SIAM Review \textbf{62}
  (2020), no.~3, 685--715.

\bibitem{memoli2020persistent}
Facundo M{\'e}moli, Zhengchao Wan, and Yusu Wang, \emph{Persistent
  {L}aplacians: properties, algorithms and implications}, 2020,
  arXiv:2012.02808.

\bibitem{meng2019weighted}
Zhenyu Meng, D.~Vijay Anand, Yunpeng Lu, Jie Wu, and Kelin Xia, \emph{Weighted
  persistent homology for biomolecular data analysis}, 2019, arXiv:1903.02890.

\bibitem{ren2018weighted}
Shiquan Ren, Chengyuan Wu, and Jie Wu, \emph{Weighted persistent homology},
  Rocky Mountain Journal of Mathematics \textbf{48} (2018), no.~8, 2661--2687.

\bibitem{rosenberg1997laplacian}
Steven Rosenberg, \emph{The {L}aplacian on a {R}iemannian manifold}, Cambridge
  University Press, 1997.

\bibitem{schaub2018random}
Michael~T. Schaub, Austin~R. Benson, Paul Horn, Gabor Lippner, and Ali
  Jadbabaie, \emph{Random walks on simplicial complexes and the normalized
  {H}odge 1-{L}aplacian}, 2018, arXiv:1807.05044.

\bibitem{smirnov2021hodgenet}
Dmitriy Smirnov and Justin Solomon, \emph{{HodgeNet}: Learning spectral
  geometry on triangle meshes}, 2021, arXiv:2104.12826.

\bibitem{wang2020persistent}
Rui Wang, Duc~Duy Nguyen, and Guo-Wei Wei, \emph{Persistent spectral graph},
  International Journal for Numerical Methods in Biomedical Engineering
  \textbf{36} (2020), no.~9, e3376.

\bibitem{warner1983foundations}
Frank~W. Warner, \emph{Foundations of differentiable manifolds and lie groups},
  Graduate Texts in Mathematics, vol.~94, Springer, 1983.

\bibitem{yang2023convolutional}
Maosheng Yang and Elvin Isufi, \emph{Convolutional learning on simplicial
  complexes}, 2023, arXiv:2301.11163.

\bibitem{yang2022simplicialfilters}
Maosheng Yang, Elvin Isufi, Michael~T. Schaub, and Geert Leus, \emph{Simplicial
  convolutional filters}, 2022, arXiv:2201.11720.

\end{thebibliography}
\end{document}